\documentclass[11pt,a4paper]{article}

\usepackage[utf8]{inputenc} 
\usepackage[T1]{fontenc} 
\usepackage{amsmath,amsfonts,amssymb,amsthm}
\usepackage{textcomp}
\usepackage{stmaryrd}
\usepackage[english]{babel}
\usepackage[a4paper]{geometry}
\usepackage{dsfont}
\usepackage{multicol}
\usepackage{enumitem}
\usepackage{tikz}

\usetikzlibrary{arrows}

\newcommand{\Z}{\mathbb{Z}}

\newcommand{\C}{\mathbb{C}}

\newcommand{\dps}{\displaystyle}

\newcommand{\pp}{\mathbb{P}}

\newcommand{\oo}{\mathcal{O}}

\newcommand{\gs}{\mathfrak{S}}

\newcommand{\Sc}{\mathbb{S}}
\newcommand{\Lb}{\mathcal{L}}
\newcommand{\Mb}{\mathcal{M}}
\newcommand{\Nb}{\mathcal{N}}
\newcommand{\fl}{\mathcal{F}\!\ell}
\newcommand{\hooklongrightarrow}{\lhook\joinrel\longrightarrow}

\DeclareMathOperator{\id}{Id}

\DeclareMathOperator{\gl}{GL}

\DeclareMathOperator{\h}{H}

\DeclareMathOperator{\Hom}{Hom}

\newtheorem{theo}{Theorem}[section]
\newtheorem{prop}[theo]{Proposition}
\newtheorem{lemma}[theo]{Lemma}
\newtheorem{cor}[theo]{Corollary}
\theoremstyle{definition}
\newtheorem{de}[theo]{Definition}
\newtheorem{rmk}[theo]{Remark}

\newenvironment{changemargin}[2]{%
\begin{list}{}{%
\setlength{\topsep}{0pt}%
\setlength{\leftmargin}{#1}%
\setlength{\rightmargin}{#2}%
\setlength{\listparindent}{\parindent}%
\setlength{\itemindent}{\parindent}%
\setlength{\parsep}{\parskip}%
}%
\item[]}{\end{list}}

\title{The Heisenberg product seen as a branching problem for connected reductive groups, stability properties}
\author{Maxime Pelletier\thanks{Univ Lyon, Université Claude Bernard Lyon 1, CNRS UMR 5208, Institut Camille Jordan, 43 blvd. du 11 novembre 1918, F-69622 Villeurbanne cedex, France (\texttt{pelletier@math.univ-lyon1.fr})}}
\date{}

\begin{document}

\maketitle

\begin{abstract}
In this article we study, in the context of complex representations of symmetric groups, some aspects of the Heisenberg product, introduced by Marcelo Aguiar, Walter Ferrer Santos, and Walter Moreira in 2017. When applied to irreducible representations, this product gives rise to the Aguiar coefficients. We prove that these coefficients are in fact also branching coefficients for representations of connected complex reductive groups. This allows to use geometric methods already developped in a previous article, notably based on notions from Geometric Invariant Theory, and to obtain some stability results on Aguiar coefficients, generalising some of the results concerning them given by Li Ying.
\end{abstract}

\section{Introduction}

The Heisenberg product was first introduced by Marcelo Aguiar, Walter Ferrer Santos, and Walter Moreira in \cite{aguiar-ferrersantos-moreira} in order to simplify and unify a diversity of related products and coproducts (e.g. Hadamard, Cauchy,
Kronecker, induction, internal, external, Solomon, composition, Malvenuto–Reutenauer, convolution...) defined on various objects: species, representations of symmetric groups, symmetric functions, endomorphisms of graded connected Hopf algebras, permutations, non–commutative symmetric functions, quasi–symmetric functions...

\vspace{5mm}

In this paper we are only interested in one of these contexts where they thus defined this Heisenberg product, namely the one of complex representations\footnote{Every representation considered throughout the article will be a complex vector space.} of symmetric groups. In this particular context, let us fix throughout the whole article two positive integers $k$ and $l$. Then we denote by $\gs_k$ (resp. $\gs_l$) the symmetric group formed by the permutations of the finite set $\{1,\dots,k\}$ (resp. $\{1,\dots,l\}$), and the Heisenberg product can be defined on two complex representations $V$ and $W$ of $\gs_k$ and $\gs_l$ respectively. It is denoted by $V\sharp W$ and is a direct sum of representations of the groups $\gs_i$ for $i$ from $\max(k,l)$ to $k+l$. We explain precisely the construction of this product in Section \ref{section_construction}.

\vspace{5mm}

One interesting thing to notice about this product is that, when $k=l$, $V\sharp W$ is a direct sum of representations of $\gs_k,\gs_{k+1},\dots,\gs_{2k}$, and the term in this direct sum which is a representation of $\gs_k$ corresponds simply to the tensor product $V\otimes W$ seen as a $\gs_k$-module (with $\gs_k$ acting diagonally). This tensor product of representations of $\gs_k$ is sometimes referred to as the ``Kronecker product'', since it gives rise to the Kronecker coefficients when applied to irreducible $\gs_k$-modules. As a consequence the Heisenberg product extends -- in a certain way -- this so-called Kronecker product.

\vspace{5mm}

An important point that we use concerning the representation theory of the symmetric groups is that the irreducible complex representations of a group $\gs_k$ are known: they are in bijection with the partitions of the integer $k$ and one can moreover construct them. If $\lambda$ is a partition of $k$ (denoted $\lambda\vdash k$), i.e. a finite non-increasing sequence $(\lambda_1\geq\lambda_2\geq\dots\geq\lambda_n>0)$ of positive integers (called parts) whose sum (sometimes called size of the partition, and denoted by $|\lambda|$) is $k$, there is an explicite construction -- several, in fact -- giving a representation of $\gs_k$ over the field $\C$ of complex numbers, which happens to be irreducible. We denote this $\gs_k$-module by $M_\lambda$. We do not detail the construction of $M_\lambda$ here: two examples of such can for instance be found in \cite[Chapter 4]{lakshmibai-brown}.

\vspace{5mm}

Considering that every complex representation of a symmetric group decomposes as a direct sum of irreducible ones, it is natural to seek to understand the Heisenberg product of two of the latter. If $\lambda$ and $\mu$ are respectively partitions of $k$ and $l$, the Heisenberg product $M_\lambda\sharp M_\mu$ is a direct sum of $\gs_i$-modules for $i\in\{\max(k,l),\dots,k+l\}$, and then every term in this sum decomposes as a direct sum of irreducible $\gs_i$-modules:
\[ M_\lambda\sharp M_\mu=\bigoplus_{i=\max(k,l)}^{k+l}\bigoplus_{\nu\vdash i}M_\nu^{\oplus a_{\lambda,\mu}^\nu}. \]
The multiplicities $a_{\lambda,\mu}^\nu$ in these decompositions are non-negative integers which are called the Aguiar coefficients. They were introduced in \cite{ying} by Li Ying, who also proved interesting stability results concerning them. We recall these results in Section \ref{section_results_ying}.

\vspace{5mm}

The fact is that Li Ying's stability results look very much like similar results already proven concerning Kronecker coefficients. Let us recall that these particular coefficients are the multiplicities $g_{\lambda,\mu,\nu}$ arising in the following decomposition:
\[ M_\lambda\otimes M_\mu=\bigoplus_{\nu\vdash k}M_\nu^{\oplus g_{\lambda,\mu,\nu}}, \]
where $\lambda$ and $\mu$ are partitions of $k$. In \cite{article1} we exposed some geometric methods allowing to prove stability results for those, as well as for some other similar coefficients. In fact these techniques can be applied as soon as we are looking at branching coefficients for complex connected reductive groups: if $G$ and $\hat{G}$ are two such groups and if we have a morphism $G\rightarrow\hat{G}$, the branching problem consists in seeing irreducible complex representations of $\hat{G}$ as $G$-modules via the previous morphism and in wondering how as such they decompose as a direct sum of irreducible ones. As a consequence, in Section \ref{section_aguiar_as_branching}, we express the Aguiar coefficients as such branching coefficients, obtaining the following result.

\begin{theo}
If $V_1$ and $V_2$ are finite dimensional complex vector spaces (of large enough dimensions), then the Aguiar coefficients are the branching coefficients for the groups $G=\gl(V_1)\times\gl(V_2)$ and $\hat{G}=\gl\big(V_1\oplus(V_1\otimes V_2)\oplus V_2\big)$, and the morphism $\varphi:G\rightarrow\hat{G}$ given by $(g_1,g_2)\mapsto\varphi_{g_1,g_2}$, where
\[ \begin{array}{rccl}
\varphi_{g_1,g_2}: & V_1\oplus(V_1\otimes V_2)\oplus V_2 & \longrightarrow & V_1\oplus(V_1\otimes V_2)\oplus V_2\\
 & u_1+(\sum_i v^{(i)}_1\otimes v^{(i)}_2)+u_2 & \longmapsto & g_1(u_1)+(\sum_i g_1(v^{(i)}_1)\otimes g_2(v^{(i)}_2))+g_2(u_2)
\end{array}. \]
\end{theo}

As a consequence we can use in Section \ref{section_new_stability_results} the same methods as in \cite{article1}, and obtain some new stability results, generalising in part those of Li Ying. The main one is:

\begin{theo}
Let $\alpha$, $\beta$, and $\gamma$ be three partitions such that, for all $d\in\Z_{>0}$, $a_{d\alpha,d\beta}^{d\gamma}=1$. Then, for all triple $(\lambda,\mu;\nu)$ of partitions, the sequence $(a_{\lambda+d\alpha,\mu+d\beta}^{\nu+d\gamma})_{d\in\Z_{\geq 0}}$ is constant for $d\gg 0$.
\end{theo}

In fact Li Ying obtained the same conclusion as in the previous theorem, for the triple $(\alpha,\beta;\gamma)=\big((1),(1);(1)\big)$. We call the triples satisfying the same property ``Aguiar-stable triples'', and we give -- also in Section \ref{section_new_stability_results} -- four other explicit examples of such ones.

\begin{prop}
The triples
\[ \big((2),(1);(2)\big),\hphantom{a}\big((2),(1);(1,1)\big),\hphantom{a}\big((2),(1);(3)\big),\text{ and }\big((2),(1);(2,1)\big) \]
are all Aguiar-stable.
\end{prop}

Finally, in Section \ref{section_bounds_stabilisation} we discuss about what we call ``bounds of stabilisation'': if $(\alpha,\beta;\gamma)$ is Aguiar-stable, such a bound is a non-negative integer $d_0$ (depending on partitions $\lambda$, $\mu$, and $\nu$) such that the sequence $(a_{\lambda+d\alpha,\mu+d\beta}^{\nu+d\gamma})_{d\in\Z_{\geq 0}}$ is constant for $d\geq d_0$. The geometric methods from \cite{article1} can here also give means to compute some bounds of stabilisation. We detail the computation for the Aguiar-stable triples $\big((1),(1);(1)\big)$, $\big((2),(1);(2)\big)$ and $\big((2),(1);(3)\big)$.

\vspace{5mm}

\noindent\textit{Acknowledgements:} I would like to thank Nicolas Ressayre for pointing out the article \cite{ying}, and for interesting discussions and advice during the preparation of this article. I also acknowledge support from the French ANR (ANR project ANR-15-CE40-0012).

\section{Definition and first properties of the Heisenberg product}

\subsection{Construction}\label{section_construction}

Let us recall that we fixed, already in the introduction, two positive integers $k$ and $l$, and that we consider the two symmetric groups $\gs_k$ and $\gs_l$.

\begin{rmk}
Notice that, for all non-negative integers $a$ and $b$, $\gs_a\times\gs_b$ can naturally be seen as a subgroup of $\gs_{a+b}$, thanks to the injective group morphism
\[ \begin{array}{rccl}
\iota_{a,b}: & \gs_a\times\gs_b & \hooklongrightarrow & \gs_{a+b}\\
 & (\sigma_a,\sigma_b) & \longrightarrow & i\mapsto\left\lbrace\begin{array}{ll}
 \sigma_a(i) & \text{if }i\in\{1,\dots,a\}\\
 a+\sigma_b(i-a) & \text{if }i\in\{a+1,\dots,a+b\}
 \end{array}\right.
\end{array}. \]
Furthermore, for any non-negative integer $a$, $\gs_a$ can be considered as a subgroup of $\gs_a\times\gs_a$ through the diagonal embedding $\Delta_a:\gs_a\hookrightarrow\gs_a\times\gs_a$.
\end{rmk}

\begin{de}
Let $V$ and $W$ be two (complex) representations of $\gs_k$ and $\gs_l$ respectively. Let $i\in\{\max(k,l),\dots,k+l\}$. One has the following diagram:
\begin{center}
\begin{tikzpicture}
\node (un) at (0,0) {$\gs_{i-l} \times \gs_{k+l-i} \times \gs_{i-k}$};
\node (deux) at (8,-2) {$\gs_i$};
\node (trois) at (0,2) {$\gs_{i-l} \times \gs_{k+l-i} \times \gs_{k+l-i} \times \gs_{i-k}$};
\node (quatre) at (8,2) {$\gs_k \times \gs_l$};
\node (cinq) at (0,-2) {$\gs_{i-l}\times\gs_l$};
\draw[right hook->] (un)--(cinq);
\draw[right hook->] (cinq)--(deux);
\draw[right hook->] (un)--(trois);
\draw[right hook->] (trois)--(quatre);
\node at (-1.8,1) {$\scriptstyle{\id_{\gs_{i-l}}\times \Delta_{k+l-i}\times \id_{\gs_{i-k}}}$};
\node at (4.4,-1.8) {$\scriptstyle{\iota_{i-l,l}}$};
\node at (5,2.2) {$\scriptstyle{\iota_{i-l,k+l-i}\times\iota_{k+l-i,i-k}}$};
\node at (-1.4,-1) {$\scriptstyle{\id_{\gs_{i-l}}\times\iota_{k+l-i,i-k}}$};
\end{tikzpicture}
\end{center}
We then consider $V\otimes W$, which is a representation of $\gs_k \times \gs_l$, and its restriction $\mathrm{Res}^{\gs_k\times\gs_l}_{\gs_{i-l} \times \gs_{k+l-i} \times \gs_{i-k}}(V\otimes W)$ to a representation of $\gs_{i-l} \times \gs_{k+l-i} \times \gs_{i-k}$. Finally we define $(V\sharp W)_i$ as the representation induced to $\gs_i$ from $\mathrm{Res}^{\gs_k\times\gs_l}_{\gs_{i-l} \times \gs_{k+l-i} \times \gs_{i-k}}(V\otimes W)$. It is then an $\gs_i$-module, and the Heisenberg product of $V$ and $W$ is
\[ V\sharp W=\bigoplus_{i=\max(k,l)}^{k+l}(V\sharp W)_i. \]
\end{de}

\noindent A remarkable thing proven in \cite{aguiar-ferrersantos-moreira} is that this product is associative.

\begin{de}
Let $\lambda\vdash k$ and $\mu\vdash l$. The Heisenberg product between the associated irreducible representations of the symmetric group decomposes as:
\[ M_\lambda\sharp M_\mu=\bigoplus_{i=\max(k,l)}^{k+l}\bigoplus_{\nu\vdash i}M_\nu^{\oplus a_{\lambda,\mu}^\nu}. \]
The coefficients $a_{\lambda,\mu}^\nu$ are called the Aguiar coefficients.
\end{de}

We adopt the convention that, if the weights of the partitions $\lambda$, $\mu$, and $\nu$ are not compatible to define an Aguiar coefficient (i.e. $|\nu|\notin\{\max(|\lambda|,|\mu|),\dots,|\lambda|+|\mu|\}$), then $a_{\lambda,\mu}^\nu=0$. Likewise, if $V$ and $W$ are respectively $\gs_k$- and $\gs_l$-modules and if $i\not\in\{\max(k,l),\dots,k+l\}$ is a positive integer, we will say that $(V\sharp W)_i$ is the trivial $\gs_i$-module $\{0\}$.

\begin{rmk}\label{remark_aguiar=kron}
As written earlier, the Heisenberg product extends the Kronecker one: when $k=l$, the lower term $(V\sharp W)_k$ of $V\sharp W$ is just $V\otimes W$ seen as a representation of $\gs_k$. As a consequence, when the three partitions $\lambda$, $\mu$, and $\nu$ have the same size, the Aguiar coefficient $a_{\lambda,\mu}^\nu$ coincides with the Kronecker coefficient $g_{\lambda,\mu,\nu}$.
\end{rmk}

\subsection{First stability results by Li Ying}\label{section_results_ying}

In this paragraph we recall some results from \cite{ying} concerning the Aguiar coefficients. One can first reformulate its main result (Theorem 2.3) as follows:
\begin{theo}[Ying]\label{main_result_ying}
Let $\lambda$ and $\mu$ be two partitions, and $i\geq\max(|\lambda|,|\mu|)$ be an integer. Then the decomposition of the $\gs_{i+d}$-module $(M_{\lambda+(d)}\sharp M_{\mu+(d)})_{i+d}$ stabilises when $d\geq 3i-|\lambda|-|\mu|-\lambda_1-\mu_1+\lambda_2+\mu_2$. Moreover, the stabilisation begins exactly at this particular integer.
\end{theo}

\begin{rmk}
In the case when the positive integer $i<\max(|\lambda|,|\mu|)$, the stabilisation of the decomposition of $(M_{\lambda+(d)}\sharp M_{\mu+(d)})_{i+d}$ is trivial: this module is $\{0\}$ for any $d\in\Z_{\geq 0}$.
\end{rmk}

From Theorem \ref{main_result_ying}, one can immediately deduce a stabilisation result for Aguiar coefficients. But this time the bound of stabilisation obtained is not optimal, and Li Ying obtains indeed in \cite[Corollary 5.2]{ying} a better one, refining this stability result for Aguiar coefficients, which can be reformulated in the following way:

\begin{prop}[Ying]\label{bound_ying}
For all partitions $\lambda$, $\mu$, and $\nu$, the sequence $\left(a_{\lambda+(d),\mu+(d)}^{\nu+(d)}\right)_{d\in\Z_{\geq 0}}$ stabilises when $d\geq\dps\frac{1}{2}\big(3|\nu|-|\lambda|-|\mu|-\lambda_1-\mu_1-\nu_1+\lambda_2+\mu_2+\nu_2-1\big)$.
\end{prop}

\section{New stability results by geometric methods}\label{stability_results_Aguiar}

\subsection{The Aguiar coefficients as branching coefficients}\label{section_aguiar_as_branching}

In order to use on the Aguiar coefficients the same methods that we used in \cite{article1}, we express these as branching coefficients for connected complex reductive groups. To do this we use a fact given in \cite[Lemma 3.3]{ying}: there is a remarkable expression of the Aguiar coefficients in terms of Littlewood-Richardson and Kronecker coefficients. We already defined Kronecker coefficients in the introduction; let us now recall a definition of Littlewood-Richardson coefficients.

\vspace{5mm}

If $V$ is a finite dimensional complex vector space, the irreducible polynomial $\gl(V)$-modules are in one to one correspondence with the partitions of length at most $\dim V$. For such a partition $\lambda$, we denote by $\Sc^\lambda V$ the corresponding irreducible representation of $\gl(V)$ ($\Sc$ is a usual notation, denoting the Schur functor). Then the Littlewood-Richardson coefficients appear in the following situation:

\begin{de}
If $\lambda$ and $\mu$ are two partitions of length at most $\dim V$, the tensor product of $\Sc^\lambda V$ and $\Sc^\mu V$ is naturally a representation of $\gl(V)$, and thus decomposes into a direct sum of irreducible ones in the following way:
\[ \Sc^\lambda V\otimes\Sc^\mu V=\bigoplus_{\nu\,|\,\ell(\nu)\leq\dim V}\Sc^\nu V^{\oplus c_{\lambda,\mu}^\nu}. \]
The multiplicities $c_{\lambda,\mu}^\nu$ are nonnegative integers called the Littlewood-Richardson coefficients.
\end{de}

Then the proposition (see \cite[Lemma 3.3]{ying}) on which the proof of Theorem \ref{main_result_ying} is strongly based is the following:
\begin{prop}[Ying]\label{prop_expr_aguiar}
For all partitions $\lambda$, $\mu$, and $\nu$,
\[ a_{\lambda,\mu}^\nu=\sum_{\alpha,\beta,\delta,\eta,\rho,\tau}c_{\alpha,\beta}^\lambda\hphantom{a}c_{\eta,\rho}^\mu\hphantom{a}g_{\beta,\eta,\delta}\hphantom{a}c_{\alpha,\delta}^\tau\hphantom{a}c_{\tau,\rho}^\nu. \]
\end{prop}

We are going to use the fact that the Kronecker and Littlewood-Richardson coefficients appear also in some classical results called ``branching rules'':

\begin{lemma}[Branching rules]\label{lemma_branching_rules}
Let $V_1$ and $V_2$ be two finite dimensional $\C$-vector spaces. We have a morphism
\[ \begin{array}{ccl}
\gl(V_1)\times\gl(V_2) & \longrightarrow & \gl(V_1\otimes V_2)\\
(g_1,g_2) & \longmapsto & \big(\sum_i v^{(i)}_1\otimes v^{(i)}_2\mapsto\sum_i g_1(v^{(i)}_1)\otimes g_2(v^{(i)}_2)\big)
\end{array}. \]
Then, for any partition $\nu$ of length at most $\dim V_1\dim V_2$, $\Sc^\nu(V_1\otimes V_2)$ can be seen as a $\gl(V_1)\times\gl(V_2)$-module via this morphism, and as such it decomposes in the following way:
\[ \Sc^\nu(V_1\otimes V_2)=\bigoplus_{\lambda,\mu}\left(\Sc^\lambda V_1\otimes\Sc^\mu V_2\right)^{\oplus g_{\lambda,\mu,\nu}} \]
(The multiplicities are indeed the Kronecker coefficients.)\\
Likewise there is a morphism
\[ \begin{array}{ccl}
\gl(V_1)\times\gl(V_2) & \longrightarrow & \gl(V_1\oplus V_2)\\
(g_1,g_2) & \longmapsto & \big(v_1+v_2\mapsto g_1(v_1)+g_2(v_2)\big)
\end{array} \]
and, for any partition $\nu$ of length at most $\dim V_1+\dim V_2$,
\[ \Sc^\nu(V_1\oplus V_2)=\bigoplus_{\lambda,\mu}\left(\Sc^\lambda V_1\otimes\Sc^\mu V_2\right)^{\oplus c_{\lambda,\mu}^\nu} \]
as representations of $\gl(V_1)\times\gl(V_2)$ once again. (The multiplicities are this time the Littlewood-Richardson coefficients.)
\end{lemma}

\begin{proof}
Some proofs of these classical facts can be found in \cite[Part 1]{sam-snowden2} (using Schur-Weyl duality): (3.11) and (3.12).
\end{proof}

Using Proposition \ref{prop_expr_aguiar} and Lemma \ref{lemma_branching_rules}, we can then see the Aguiar coefficients as branching coefficients for connected complex reductive groups:

\begin{theo}\label{theorem_aguiar_branching}
Let $V_1$ and $V_2$ be two finite dimensional $\C$-vector spaces, and denote $G=\gl(V_1)\times\gl(V_2)$ and $\hat{G}=\gl\big(V_1\oplus(V_1\otimes V_2)\oplus V_2\big)$. We consider the morphism $\varphi:G\rightarrow\hat{G}$, defined by $(g_1,g_2)\mapsto\varphi_{g_1,g_2}$, where
\[ \begin{array}{rccl}
\varphi_{g_1,g_2}: & V_1\oplus(V_1\otimes V_2)\oplus V_2 & \longrightarrow & V_1\oplus(V_1\otimes V_2)\oplus V_2\\
 & u_1+(\sum_i v^{(i)}_1\otimes v^{(i)}_2)+u_2 & \longmapsto & g_1(u_1)+(\sum_i g_1(v^{(i)}_1)\otimes g_2(v^{(i)}_2))+g_2(u_2)
\end{array}. \]
Then the Aguiar coefficients are the branching coefficients for this situation. In other words, if $\nu$ is a partition such that $\ell(\nu)\leq\dim\big(V_1\oplus (V_1\otimes V_2)\oplus V_2\big)$, then
\[ \Sc^\nu\big(V_1\oplus(V_1\otimes V_2)\oplus V_2\big)=\bigoplus_{\lambda,\mu}\left(\Sc^\lambda V_1\otimes\Sc^\mu V_2\right)^{\oplus a_{\lambda,\mu}^\nu} \]
(as representations of $\gl(V_1)\times\gl(V_2)$).
\end{theo}

\begin{proof}
Let $\nu$ be as above. Then, using the definition of the Littlewood-Richardson coefficients as well as Lemma \ref{lemma_branching_rules}:
\[ \begin{array}{rcl}
\Sc^\nu\big(V_1\oplus(V_1\otimes V_2)\oplus V_2\big) & = & \dps\bigoplus_{\tau,\rho}c_{\tau,\rho}^\nu\hphantom{a}\Sc^\tau\big(V_1\oplus(V_1\otimes V_2)\big)\otimes\Sc^\rho V_2\\
 & = & \dps\bigoplus_{\tau,\rho,\alpha,\delta}c_{\tau,\rho}^\nu\hphantom{a}c_{\alpha,\delta}^\tau\hphantom{a}\Sc^\alpha V_1\otimes\Sc^\delta(V_1\otimes V_2)\otimes\Sc^\rho V_2\\
 & = & \dps\bigoplus_{\tau,\rho,\alpha,\delta,\beta,\eta}c_{\tau,\rho}^\nu\hphantom{a}c_{\alpha,\delta}^\tau\hphantom{a}g_{\beta,\eta,\delta}\hphantom{a}\Sc^\alpha V_1\otimes\Sc^\beta V_1\otimes\Sc^\eta V_2\otimes\Sc^\rho V_2\\
 & = & \dps\bigoplus_{\tau,\rho,\alpha,\delta,\beta,\eta,\lambda}c_{\tau,\rho}^\nu\hphantom{a}c_{\alpha,\delta}^\tau\hphantom{a}g_{\beta,\eta,\delta}\hphantom{a}c_{\alpha,\beta}^\lambda\hphantom{a}\Sc^\lambda V_1\otimes\Sc^\eta V_2\otimes\Sc^\rho V_2\\
 & = & \dps\bigoplus_{\tau,\rho,\alpha,\delta,\beta,\eta,\lambda,\mu}c_{\tau,\rho}^\nu\hphantom{a}c_{\alpha,\delta}^\tau\hphantom{a}g_{\beta,\eta,\delta}\hphantom{a}c_{\alpha,\beta}^\lambda\hphantom{a}c_{\eta,\rho}^\mu\hphantom{a}\Sc^\lambda V_1\otimes\Sc^\mu V_2\\
 & = & \dps\bigoplus_{\lambda,\mu}a_{\lambda,\mu}^\nu\hphantom{a}\Sc^\lambda V_1\otimes\Sc^\mu V_2,
\end{array} \]
using Proposition \ref{prop_expr_aguiar}.
\end{proof}

We can use this interpretation of the Aguiar coefficients to have a geometric point of view on them. For this reason we have to recall a result known as the Borel-Weil Theorem. Consider $V$ a $\C$-vector space of dimension $n\in\Z_{>0}$. If $B$ is a Borel subgroup of $\gl(V)$, then the complete flag variety $\fl(V)=\{E_1\subset E_2\subset\dots\subset E_{n-1}\,|\,\forall i, \; \dim(E_i)=i\}$ of $V$ is naturally isomorphic to $\gl(V)/B$. Moreover, any $n$-tuple $\alpha=(\alpha_1,\dots,\alpha_n)$ of integers defines uniquely a character of $B$, and we denote by $\C_\alpha$ the associated one-dimensional complex representation of $B$. As a consequence, any partition $\lambda$ of length at most $n$ allows to define the following fibre product
\[ \Lb_\lambda=\gl(V)\times_B\C_{-\lambda}, \]
which is a line bundle on $\gl(V)/B\simeq\fl(V)$ on which $\gl(V)$ acts (by left multiplication). Then the Borel-Weil Theorem states that the space of sections $\h^0(X,\Lb_\lambda)$ is a $\gl(V)$-module isomorphic to the dual of the irreducible representation $\Sc^\lambda V$.

\begin{cor}\label{cor_aguiar_sections_line_bundle}
Let $\lambda$, $\mu$, $\nu$ be three partitions. Taking $V_1$ and $V_2$ as in the previous theorem, we set:
\[ G=\gl(V_1)\times\gl(V_2), \]
\[ X=\fl(V_1)\times\fl(V_2)\times\fl\big(V_1\oplus(V_1\otimes V_2)\oplus V_2\big), \]
and
\[ \Lb=\Lb_\lambda\otimes\Lb_\mu\otimes\Lb_\nu^* \]
($G$-linearised line bundle on $X$). Then
\[ a_{\lambda,\mu}^\nu=\dim\h^0(X,\Lb)^G. \]
\end{cor}

\begin{proof}
Using Theorem \ref{theorem_aguiar_branching},
\[ \Sc^\nu\big(V_1\oplus(V_1\otimes V_2)\oplus V_2\big)=\bigoplus_{\lambda,\mu}\left(\Sc^\lambda V_1\otimes\Sc^\mu V_2\right)^{\oplus a_{\lambda,\mu}^\nu} \]
as representations of $G=\gl(V_1)\times\gl(V_2)$. As a consequence, Schur's Lemma implies that
\[ a_{\lambda,\mu}^\nu=\dim\Big((\Sc^\lambda V_1)^*\otimes(\Sc^\mu V_2)^*\otimes\Sc^\nu\big(V_1\oplus(V_1\otimes V_2)\oplus V_2\big)\Big)^G, \]
and we immediately get the conclusion by using three times the Borel-Weil Theorem.
\end{proof}

\subsection{Consequences and new examples of stable triples}\label{section_new_stability_results}

Since the Aguiar coefficients can be expressed as $\dim\h^0(X,\Lb)^G$, for well-chosen $G$, $X$, and $\Lb$ (cf Corollary \ref{cor_aguiar_sections_line_bundle}), the same techniques used in \cite{article1} for Kronecker coefficients apply. The main notion we use comes from Geometric Invariant Theory and is the one of ``semi-stable points'': if $X$ is a projective algebraic variety, on which acts a connected complex reductive group $G$, and if $\Lb$ is a $G$-linearised line bundle on $X$, then the set of semi-stable points in $X$ relatively to $\Lb$ is
\[ X^{ss}(\Lb)=\{x\in X\,|\,\exists d\in\Z_{>0}, \; \exists\sigma\in\h^0(X,\Lb^{\otimes d})^G, \; \sigma(x)\neq 0\}. \]
The set of unstable points relatively to $\Lb$ is its complementary: $X^{us}(\Lb)=X\setminus X^{ss}(\Lb)$. This geometric point of view allows to obtain the following:

\begin{theo}\label{thm_stab_aguiar}
Let $\alpha$, $\beta$, and $\gamma$ be three partitions such that, for all $d\in\Z_{>0}$, $a_{d\alpha,d\beta}^{d\gamma}=1$. Then, for all triple $(\lambda,\mu;\nu)$ of partitions, the sequence $(a_{\lambda+d\alpha,\mu+d\beta}^{\nu+d\gamma})_{d\in\Z_{\geq 0}}$ is constant for $d\gg 0$.
\end{theo}

\begin{proof}
We give a sketch of the proof. For every details see \cite{article1}. With Corollary \ref{cor_aguiar_sections_line_bundle}, we can write, with $V_1$ and $V_2$ vector spaces of large enough dimension, $X=\fl(V_1)\times\fl(V_2)\times\fl\big(V_1\oplus(V_1\otimes V_2)\oplus V_2\big)$, $G=\gl(V_1)\times\gl(V_2)$, $\Lb=\Lb_\alpha\otimes\Lb_\beta\otimes\Lb_\gamma^*$, $\Mb=\Lb_\lambda\otimes\Lb_\mu\otimes\Lb_\nu^*$:
\[ \forall d\in\Z_{\geq 0}, \; a_{\lambda+d\alpha,\mu+d\beta}^{\nu+d\gamma}=\dim\h^0\left(X,\Mb\otimes\Lb^{\otimes d}\right)^G. \]
Then a result by Victor Guillemin and Shlomo Sternberg (see \cite[Proposition 2.4]{article1}) gives
\[ \h^0\left(X,\Mb\otimes\Lb^{\otimes d}\right)^G\simeq\h^0\left(X^{ss}(\Mb\otimes\Lb^{\otimes d}),\Mb\otimes\Lb^{\otimes d}\right)^G. \]
Moreover there exists $d_0\in\Z_{\geq 0}$ such that $X^{ss}(\Mb\otimes\Lb^{\otimes d})\subset X^{ss}(\Lb)$ as soon as $d\geq d_0$ (see \cite[Proposition 2.7]{article1}). Hence, when $d\geq d_0$,
\[ \h^0\left(X,\Mb\otimes\Lb^{\otimes d}\right)^G\simeq\h^0\left(X^{ss}(\Lb),\Mb\otimes\Lb^{\otimes d}\right)^G. \]
Now, since $\h^0(X,\Lb^{\otimes d})^G\simeq\C$ for any $d\in\Z_{>0}$, we can use a corollary of Luna's étale slice Theorem (see \cite[Section 2.3]{article1}): $X^{ss}(\Lb)\simeq G\times_H S$ with $H$ a reductive subgroup of $G$ and $S$ a finite dimensional $\C$-vector space on which $H$ acts linearly. As a consequence, if $d\geq d_0$,
\[ \begin{array}{rcl}
\h^0\left(X,\Mb\otimes\Lb^{\otimes d}\right)^G & \simeq & \h^0\left(G\times_H S,\Mb\otimes\Lb^{\otimes d}\right)^G\\
 & \simeq & \h^0\left(S,\Mb\otimes\Lb^{\otimes d}\right)^H\\
 & \simeq & \h^0(S,\Mb)^H
\end{array} \]
since $\Lb$ is trivial as a $H$-linearised line bundle on $S$ (see \cite[Proposition 2.8]{article1}).
\end{proof}

\begin{de}
A triple $(\alpha,\beta;\gamma)$ of partitions such that $a_{\alpha,\beta}^\gamma\neq 0$ and that, for all triple $(\lambda,\mu;\nu)$ of partitions, $(a_{\lambda+d\alpha,\mu+d\beta}^{\nu+d\gamma})_{d\in\Z_{\geq 0}}$ is constant for $d\gg 0$ is said to be Aguiar-stable.
\end{de}

With Theorem \ref{thm_stab_aguiar} we re-obtain immediately Li Ying's result on the stabilisation of the Aguiar coefficients (minus the bound of stabilisation), which can be reformulated as follows:
\begin{cor}
The triple $\big((1),(1);(1)\big)$ is Aguiar-stable.
\end{cor}

\begin{proof}
For all $d\in\Z_{>0}$, according to Remark \ref{remark_aguiar=kron}, $a_{(d),(d)}^{(d)}=g_{(d),(d),(d)}=1$.
\end{proof}

\begin{rmk}
On a more general note, the same reasoning shows that every stable triple (i.e. same as Aguiar-stable but in the sense of Kronecker coefficients) is Aguiar-stable. For results producing stable triples, see \cite{stembridge}, \cite{manivel}, \cite{vallejo}, and \cite{article2}.
\end{rmk}

We can also give some new explicit examples of ``small'' Aguiar-stable triples:

\begin{prop}
The triples
\[ \big((2),(1);(2)\big),\hphantom{a}\big((2),(1);(1,1)\big),\hphantom{a}\big((2),(1);(3)\big),\text{ and }\big((2),(1);(2,1)\big) \]
are all Aguiar-stable triples.
\end{prop}

\begin{proof}
Let us write the proof in detail for $\big((2),(1);(2)\big)$, for instance. The three other ones work similarly. Let $d\in\Z_{>0}$. Then
\[ a_{d(2),d(1)}^{d(2)}=\sum_{\alpha,\rho,\tau,\beta,\eta,\delta}c_{\alpha,\beta}^{(2d)}\hphantom{a}c_{\eta,\rho}^{(d)}\hphantom{a}g_{\beta,\eta,\delta}\hphantom{a}c_{\alpha,\delta}^\tau\hphantom{a}c_{\tau,\rho}^{(2d)}. \]
But the Littlewood-Richardson rule (see for instance \cite[Section 5]{fulton3}) shows that the coefficient $c_{\alpha,\beta}^{(2d)}$ is zero unless $\alpha$ and $\beta$ have only one part, and $|\alpha|+|\beta|=2d$ (and then this coefficient is 1). As a consequence,
\[ a_{d(2),d(1)}^{d(2)}=\sum_{\begin{array}{c}\scriptstyle{\rho,\tau,\eta,\delta,}\\\scriptstyle{n\in\llbracket 0,2d\rrbracket}\end{array}}c_{\eta,\rho}^{(d)}\hphantom{a}g_{(n),\eta,\delta}\hphantom{a}c_{(2d-n),\delta}^\tau\hphantom{a}c_{\tau,\rho}^{(2d)}. \]
The same is true for the coefficient $c_{\eta,\rho}^{(d)}$ and the partitions $\eta$ and $\rho$. So
\[ a_{d(2),d(1)}^{d(2)}=\sum_{\begin{array}{c}\scriptstyle{\tau,\delta,}\\\scriptstyle{n\in\llbracket 0,2d\rrbracket}\\\scriptstyle{m\in\llbracket 0,d\rrbracket}\end{array}}g_{(n),(d-m),\delta}\hphantom{a}c_{(2d-n),\delta}^\tau\hphantom{a}c_{\tau,(m)}^{(2d)}. \]
And then the Kronecker coefficient $g_{(n),(d-m),\delta}$ is zero unless $n=d-m$. Moreover, if this is verified, $g_{(n),(n),\delta}$ is zero unless $\delta=(n)$ (and then this coefficient is 1). Hence
\[ a_{d(2),d(1)}^{d(2)}=\sum_{\begin{array}{c}\scriptstyle{\tau,}\\\scriptstyle{n\in\llbracket 0,d\rrbracket}\end{array}}c_{(2d-n),(n)}^\tau\hphantom{a}c_{\tau,(d-n)}^{(2d)}. \]
The coefficient $c_{(2d-n),(n)}^\tau$ is then zero unless $|\tau|=2d$. Furthermore, the other coefficient $c_{\tau,(d-n)}^{(2d)}$ is zero unless $|\tau|=2d-d+n=d+n$. So
\[ a_{d(2),d(1)}^{d(2)}=\sum_{\tau\vdash 2d}c_{(2d-d),(d)}^\tau\hphantom{a}c_{\tau,(d-d)}^{(2d)}=\sum_{\tau\vdash 2d}c_{(d),(d)}^\tau\hphantom{a}c_{\tau,(0)}^{(2d)}. \]
Finally this product is zero unless $\tau=(2d)$ (by the Littlewood-Richardson rule, for instance). Thus
\[ a_{d(2),d(1)}^{d(2)}=c_{(d),(d)}^{(2d)}\hphantom{a}c_{(2d),(0)}^{(2d)}=1, \]
and $\big((2),(1);(2)\big)$ is Aguiar-stable by Theorem \ref{thm_stab_aguiar}.
\end{proof}

\section{Some explicit bounds of stabilisation}\label{section_bounds_stabilisation}

\begin{de}
When $(\alpha,\beta;\gamma)$ is an Aguiar-stable triple, a bound of stabilisation for $(\alpha,\beta,\gamma)$ is, for any triple $(\lambda,\mu;\nu)$ of partitions, an integer $d_0\in\Z_{\geq 0}$ (depending on $\lambda$, $\mu$, and $\nu$) such that $a_{\lambda+d\alpha,\mu+d\beta}^{\nu+d\gamma}$ is constant for $d\geq d_0$.
\end{de}

In this section we are going to compute bounds of stabilisation for three examples of Aguiar-stable triples: $\big((1),(1);(1)\big)$, $\big((2),(1);(2)\big)$, and $\big((2),(1);(3)\big)$. The proof of Theorem \ref{thm_stab_aguiar} gives us a sufficient condition to obtain them: let us fix from now on an Aguiar-stable triple $(\alpha,\beta;\gamma)$ (we will specialise this triple later) and a triple $(\lambda,\mu;\nu)$ of partitions. We also consider vector spaces $V_1$ and $V_2$ as before (of dimension at least 2), and denote $V=V_1\oplus(V_1\otimes V_2)\oplus V_2$, such that
\[ a_{\alpha,\beta}^\gamma(=1)=\dim\h^0(X,\Lb)^G \qquad\text{and}\qquad a_{\lambda,\mu}^\nu=\dim\h^0(X,\Mb)^G, \]
with $G=\gl(V_1)\times\gl(V_2)$, $X=\fl(V_1)\times\fl(V_2)\times\fl(V)$, $\Lb=\Lb_\alpha\otimes\Lb_\beta\otimes\Lb_\gamma^*$, and $\Mb=\Lb_\lambda\otimes\Lb_\mu\otimes\Lb_\nu^*$. We fix finally a basis $\underline{e}=(e_1,\dots,e_{n_1})$ of $V_1$ and a basis $\underline{f}=(f_1,\dots,f_{n_2})$ of $V_2$. We now know that every $d_0\in\Z_{\geq 0}$ such that
\[ \forall d\geq d_0, \: X^{ss}\left(\Mb\otimes\Lb^{\otimes d}\right)\subset X^{ss}(\Lb) \]
is a bound of stabilisation.

\vspace{5mm}

One important tool for our computation is a numerical criterion of semi-stability known as the Hilbert-Mumford criterion:

\begin{de}
Let $Y$ be a projective variety on which a connected reductive group $H$ acts, and $\Nb$ an $H$-linearised line bundle on $Y$. Let $y\in Y$ and $\tau$ be a one-parameter subgroup of $H$ (denoted $\tau\in X_*(H)$). Since $Y$ is projective, $\lim\limits_{t\to 0}\tau(t).y$ exists. We denote it by $z$. This point is fixed by the image of $\tau$, and so $\C^*$ acts via $\tau$ on the fibre $\Nb_z$. Then there exists an integer $\mu^\Nb(y,\tau)$ such that, for all $t\in\C^*$ and $\tilde{z}\in\Nb_z$,
\[ \tau(t).\tilde{z}=t^{-\mu^\Nb(y,\tau)}\tilde{z}. \]
\end{de}

\noindent The Hilbert-Mumford criterion can then be stated as (see e.g. \cite[Lemma 2]{ressayre}):

\begin{prop}[Hilbert-Mumford Criterion]
In the settings of the previous definition, if in addition $\Nb$ is semi-ample, then:
\[ y\in Y^{ss}(\Nb)\quad\Longleftrightarrow\quad\forall\tau\in X_*(H), \; \mu^\Nb(y,\tau)\leq 0. \]
\end{prop}

\noindent Following this property, a one-parameter subgroup $\tau$ such that $\mu^\Nb(y,\tau)$ will be said ``destabilising'' for $y$ (relatively to $\Nb$).

\vspace{5mm}

We will then begin the computation by considering the projection:
\vspace{-5mm}
\begin{changemargin}{-2mm}{-2mm}
\[ \begin{array}{rccl}
\pi: & X & \longrightarrow & \overline{X}=\pp(V_1)\times\pp(V_2)\times\pp(V^*)\\
 & \big((W_{1,i})_i,(W_{2,i})_i,(W'_i)_i\big) & \longmapsto & \big(W_{1,1},W_{2,1},\{\varphi\in V^*\,|\,\ker\varphi=W'_{n_1n_2+n_1+n_2-1}\}\big)
\end{array}. \]
\end{changemargin}
We also denote by $\overline{\Lb}$ the ample line bundle on $\overline{X}$ whose pull-back by $\pi$ is $\Lb$, by $\underline{e}^*=(e_1^*,\dots,e_{n_1}^*)$ and $\underline{f}^*=(f_1^*,\dots,f_{n_2}^*)$ the dual bases of $\underline{e}$ and $\underline{f}$ respectively, and set $n=\min(n_1,n_2)$.

\begin{prop}\label{prop_orbits}
Set $\varphi_n=\sum_{i=1}^n e_i^*\otimes f_i^*\in V_1^*\otimes V_2^*\simeq(V_1\otimes V_2)^*$. The $G$-orbit $\oo_0$ of $\overline{x}_0=\big(\C e_1,\C f_1,\C(e_1^*+e_{n_1}^*+f_1^*+f_{n_2}^*+\varphi_n)\big)\in\overline{X}$ is open in $\overline{X}$. Moreover, if we denote respectively by $\oo_1$, $\oo_2$, and $\oo_3$ the $G$-orbits in $\overline{X}$ of
\[ \overline{x}_1=\big(\C e_1,\C f_2,\C(e_1^*+e_{n_1}^*+f_2^*+f_{n_2}^*+\varphi_n)\big), \]
\[ \overline{x}_2=\big(\C e_1,\C f_1,\C(e_{n_1}^*+f_1^*+f_{n_2}^*+\varphi_n)\big), \]
and
\[ \overline{x}_3=\big(\C e_1,\C f_1,\C(e_1^*+e_{n_1}^*+f_{n_2}^*+\varphi_n)\big), \]
then
\[ \overline{\oo_1}\cup\overline{\oo_2}\cup\overline{\oo_3}=\Big\{\big(\C v_1,\C v_2,\C(\underset{\in V_1^*}{\underbrace{\varphi_1}}+\underset{\in V_2^*}{\underbrace{\varphi_2}}+\underset{\in(V_1\otimes V_2)^*}{\underbrace{\varphi}})\big)\in\overline{X}\,|\,\varphi_1(v_1)\varphi_2(v_2)\linebreak\varphi(v_1\otimes v_2)=0\Big\}. \]
In addition, among $\{\oo_1,\oo_2,\oo_3\}$, no orbit is contained in the closure of another one.
\end{prop}

\begin{proof}
We consider an element $\big(\C v_1,\C v_2,\C(\varphi_1+\varphi_2+\varphi)\big)\in\overline{X}$. Similarly to the proof of Proposition 3.3 from \cite{article1}, we are only interested in the orbits of maximal dimension or of dimension just below that. Then, considering the usual isomorphism $(V_1\otimes V_2)^*\simeq\Hom(V_1,V_2^*)$, we say that $\varphi$ corresponds to a linear map $\varphi':V_1\rightarrow V_2^*$, on which $G$ acts by conjugation. As a consequence we only need to consider the case when $\varphi'$ is of maximal rank ($n$, that is), since all the orbits with $\varphi'$ of lower rank will be contained in the closure of such an orbit.

\vspace{5mm}

Thus we rather consider an element $\overline{x}=\big(\C v_1,\C v_2,\C(\varphi_1+\varphi_2+\varphi_n)\big)\in\overline{X}$, with $\varphi_n=\sum_{i=1}^n e_i^*\otimes f_i^*\in V_1^*\otimes V_2^*\simeq(V_1\otimes V_2)^*$, corresponding to a linear map $\varphi'_n:V_1\rightarrow V_2^*$. Then the linear maps $\varphi_1,\varphi_2,\varphi'_n,\varphi_n$, together with the vectors $v_1$ and $v_2$, give some vector subspaces of $V_1$, $V_2$, and $V_1\otimes V_2$, whose relative positions will give us descriptions of the orbits we are interested in:
\begin{itemize}
\item in $V_1$ : $\C v_1$, $\ker\varphi'_n\subset(\varphi'_n)^{-1}\big((\C v_2)^\perp\big)$, and $\ker\varphi_1$;
\item in $V_2$ : $\C v_2$, $\ker{}^t\varphi'_n\subset({}^t\varphi'_n)^{-1}\big((\C v_1)^\perp\big)$, and $\ker\varphi_2$;
\item in $V_1\otimes V_2$ : $\C v_1\otimes v_2$ and $\ker\varphi_n$.
\end{itemize}

\vspace{5mm}

Then we see that there is an open orbit, $\oo_0$, characterised by:
\begin{itemize}
\item $\varphi_1(v_1)\neq 0$, $\varphi'_n(v_1)\neq 0$, $\ker\varphi'_n\not\subset\ker\varphi_1$ (or rather, if $n=n_1$, $(\varphi'_n)^{-1}\big((\C v_2)^\perp\big)\not\subset\ker\varphi_1$), and $\ker\varphi_1\not\subset(\varphi'_n)^{-1}\big((\C v_2)^\perp\big)$;
\item $\varphi_2(v_2)\neq 0$, ${}^t\varphi'_n(v_2)\neq 0$, $\ker{}^t\varphi'_n\not\subset\ker\varphi_2$ (or rather, if $n=n_2$, $({}^t\varphi'_n)^{-1}\big((\C v_1)^\perp\big)\not\subset\ker\varphi_2$), and $\ker\varphi_2\not\subset({}^t\varphi'_n)^{-1}\big((\C v_1)^\perp\big)$;
\item $\varphi_n(v_1\otimes v_2)\neq 0$.
\end{itemize}
And the point $\overline{x}_0$ given above verifies all these conditions.

\vspace{5mm}

Finally the subset $\Big\{\big(\C v_1,\C v_2,\C(\varphi_1+\varphi_2+\varphi)\big)\in\overline{X}\,|\,\varphi_1(v_1)\varphi_2(v_2)\varphi(v_1\otimes v_2)=0\Big\}$ can be written as $\overline{\oo_1}\cup\overline{\oo_2}\cup\overline{\oo_3}$ for three orbits $\oo_1$, $\oo_2$, and $\oo_3$, characterised by the same equations as $\oo_0$ except for:
\begin{itemize}
\item $\varphi_n(v_1\otimes v_2)=0$ for $\oo_1$,
\item $\varphi_1(v_1)=0$ for $\oo_2$,
\item $\varphi_2(v_2)=0$ for $\oo_3$.
\end{itemize}
Then it is easy to check that $\overline{x}_1\in\oo_1$, $\overline{x}_2\in\oo_2$, and $\overline{x}_3\in\oo_3$.
\end{proof}

\subsection{Murnaghan case and comparison with the results by Li Ying}

We deal in this section with the case of $(\alpha,\beta;\gamma)=\big((1),(1);(1)\big)$, which is the one treated by Li Ying in \cite{ying}. Since the triple $\big((1),(1),(1)\big)$ is also a stable triple for Kronecker coefficients, well-studied and first observed by Francis Murnaghan in \cite{murnaghan}, we often refer to it as the ``Murnaghan case''. In that case $\overline{\Lb}=\oo(1)\otimes\oo(1)\otimes\oo(1)$. Moreover, since $\dim\h^0(\overline{X},\overline{\Lb}^{\otimes d})^G=1$ for any $d\in\Z_{>0}$ and since $\big(\C v_1,\C v_2,\C(\varphi_1+\varphi_2+\varphi)\big)\in\overline{X}\mapsto\varphi(v_1\otimes v_2)$ gives a non-zero $G$-invariant section of $\overline{\Lb}$ on $\overline{X}$,
\[ \overline{X}^{us}(\overline{\Lb})=\Big\{\big(\C v_1,\C v_2,\C(\varphi_1+\varphi_2+\varphi)\big)\in\overline{X}\,|\,\varphi(v_1\otimes v_2)=0\Big\}. \]
Thus, according to Proposition \ref{prop_orbits} (and its proof),
\[ \overline{X}^{us}(\overline{\Lb})=\overline{\oo_1}=\overline{G.\overline{x}_1}. \]

\vspace{5mm}

In the group $G=\gl(V_1)\times\gl(V_2)$ we consider the maximal torus $T$ of elements whose matrices in the bases $\underline{e}$ and $\underline{f}$ are diagonal. Let us give a practical notation for one-parameter subgroups of $T$: such a subgroup is of the form
\[ \begin{array}{rccl}
\tau: & \C^* & \longrightarrow & T\\
 & t & \longmapsto & (\begin{pmatrix}
 t^{a_1} &&&\\
 & t^{a_2} &&\\
 && \ddots &\\
 &&& t^{a_{n_1}}
 \end{pmatrix},\begin{pmatrix}
 t^{b_1} &&&\\
 & t^{b_2} &&\\
 && \ddots &\\
 &&& t^{b_{n_2}}
 \end{pmatrix}
\end{array}, \]
with $a_1,\dots,a_{n_1},b_1,\dots,b_{n_2}$ integers, and will be denoted by $\tau=(a_1,\dots,a_{n_1}\,|\,b_1,\dots,b_{n_2})$.

\vspace{5mm}

Then we see that the one-parameter subgroup $\tau_1=\left.\big(1,0,\dots,0\text{ }\right|0,1,0,\dots,0\big)$ of $T$ (hence of $G$) is destabilising for $\overline{x}_1$ : $\mu^{\overline{\Lb}}(\overline{x}_1,\tau_1)=1$. Moreover a -- not so complicated -- calculation (see \cite[Section 3.2.4]{article1} for details on a computation completely similar) yields:
\[ \max_{x\in\pi^{-1}(\overline{x}_1)}\left(-\mu^\Mb(x,\tau_1)\right)=-\lambda_1-\mu_1+\nu_1+2\nu_2+\sum_{k=1}^{n_1+n_2-1}\nu_{k+2}. \]
It follows that:
\begin{theo}\label{thm_bound_murnaghan_aguiar}
The sequence $\left(a_{\lambda+(d),\mu+(d)}^{\nu+(d)}\right)_{d\in\Z_{\geq 0}}$ is constant when $d\geq -\lambda_1-\mu_1+\nu_1+2\nu_2+\dps\sum_{k=1}^{n_1+n_2-1}\nu_{k+2}$.
\end{theo}

\begin{proof}
Set $d_0=-\lambda_1-\mu_1+\nu_1+2\nu_2+\dps\sum_{k=1}^{n_1+n_2-1}\nu_{k+2}$. Then, for all $x\in\pi^{-1}(\overline{x}_1)$ and all $d>d_0$,
\[ \mu^{\Mb\otimes\Lb^{\otimes d}}(x,\tau_1)=\mu^\Mb(x,\tau_1)+d\mu^{\overline{\Lb}}(\overline{x}_1,\tau_1)=\mu^\Mb(x,\tau_0)+d>0 \]
because $d>d_0\geq -\mu^\Mb(x,\tau_1)$. Thus, by the Hilbert-Mumford criterion, $x\in X^{us}\left(\Mb\otimes\Lb^{\otimes d}\right)$, and $\pi^{-1}(G.\overline{x}_1)\subset X^{us}\left(\Mb\otimes\Lb^{\otimes d}\right)$ since $\pi$ is $G$-equivariant. As a consequence, since $\pi^{-1}(G.\overline{x}_1)$ is dense in $X^{us}(\Lb)$ (because $G.\overline{x}_1$ is dense in $\overline{X}^{us}(\overline{\Lb})$) and since $X^{us}\left(\Mb\otimes\Lb^{\otimes d}\right)$ is closed, $X^{us}(\Lb)\subset X^{us}\left(\Mb\otimes\Lb^{\otimes d}\right)$. That shows why the sequence $\left(a_{\lambda+(d),\mu+(d)}^{\nu+(d)}\right)_{d\in\Z_{\geq 0}}$ is constant when $d>d_0$. To justify the fact that it is constant as soon as $d=d_0$, we use an argument of quasipolynomiality detailed in \cite[Section 3.4]{article1}.
\end{proof}

\paragraph{Recovering Li Ying's bound with our method:} It is possible by choosing a different one-parameter subgroup destabilising $\overline{x}_1$: if we consider the one-parameter subgroup
\[ \tau'_1=\left.\big(2,0,1,\dots,1\text{ }\right|0,2,1,\dots,1\big), \]
it destabilises $\overline{x}_1$: $\mu^{\overline{\Lb}}(\overline{x}_1,\tau'_1)=2$. Moreover,
\begin{changemargin}{-4mm}{-4mm}
\[ \begin{array}{rcl}
\dps\max_{x\in\pi^{-1}(\overline{x}_1)}\left(-\dps\mu^\Mb(x,\tau'_1)\right) & = & -2\lambda_1-\lambda_3-\dots-\lambda_{n_1}-2\mu_1-\mu_3-\dots-\mu_{n_2}+2\nu_1+4\nu_2\\
 & & +3(\nu_3+\dots+\dps\nu_{n_1+n_2-2}) +2(\nu_{n_1+n_2-1}+\dots+\nu_{n_1n_2-n_1-n_2+5})\\
 & & +\dps\nu_{n_1n_2-n_1-n_2+6}+\dots+\nu_{n_1n_2+n_1+n_2-3},
\end{array} \]
\end{changemargin}
which gives even a slight improvement of Li Ying's bound for ``long'' partitions $\nu$ (i.e. of length $>n_1+n_2-2$), according to the previous expression of this bound (see Proposition \ref{bound_ying}).

\paragraph{Examples:} The bound of Theorem \ref{thm_bound_murnaghan_aguiar} is for instance 15 for the triple $\big((7,3),\linebreak(5,4,2);(6,6,5,4)\big)$, whereas Li Ying's (cf Corollary \ref{bound_ying}) is 18 (note that the improved bound obtained right above is 17). But the contrary is also possible: for the triple $\big((3,3,3),(4,3,2,1);(5,4,1)\big)$, Li Ying's bound is 4 whereas ours is 7.

\subsection{Two other cases}

\noindent \textbf{For $\boldsymbol{(\alpha,\beta;\gamma)=\big((2),(1);(2)\big)}$:} Then $\overline{\Lb}=\oo(2)\otimes\oo(1)\otimes\oo(2)$ and a non-zero $G$-invariant section of $\overline{\Lb}$ on $\overline{X}$ is given by $\C(\varphi_1+\varphi_2+\varphi)\big)\in\overline{X}\mapsto\varphi_1(v_1)\varphi(v_1\otimes v_2)$. As a consequence,
\[ \overline{X}^{us}(\overline{\Lb})=\{\big(\C v_1,\C v_2,\C(\varphi_1+\varphi_2+\varphi)\big)\in\overline{X}\,|\,\varphi_1(v_1)\varphi(v_1\otimes v_2)=0\}=\overline{\oo_1}\cup\overline{\oo_2}, \]
thanks to Proposition \ref{prop_orbits} and its proof.\\
Then we take the same $\tau_1$ as before to destabilise $\overline{x}_1$ (still $\mu^{\overline{\Lb}}(\overline{x}_1,\tau_1)=1$), and $\tau_2=\left.\big(1,0,\dots,0\text{ }\right|-1,0,\dots,0\big)$ which destabilises $\overline{x}_2$: $\mu^{\overline{\Lb}}(\overline{x}_2,\tau_2)=1$. Finally,
\[ \max_{x\in\pi^{-1}(\overline{x}_2)}\left(-\mu^\Mb(x,\tau_2)\right)=-\lambda_1+\mu_1+\sum_{k=1}^{n_2}\nu_{k+1}-\sum_{k=1}^{n_1}\nu_{n_1n_2+n_1+n_2+1-k}. \]

\begin{theo}
The sequence $\left(a_{\lambda+(2d),\mu+(d)}^{\nu+(2d)}\right)_{d\in\Z_{\geq 0}}$ is constant when
\[ d\geq -\lambda_1+ \max\left(-\mu_1+\nu_1+2\nu_2+\dps\sum_{k=1}^{n_1+n_2-1}\nu_{k+2},\mu_1+\sum_{k=1}^{n_2}\nu_{k+1}- \sum_{k=1}^{n_1}\nu_{n_1n_2+n_1+n_2+1-k}\right). \]
\end{theo}

\begin{proof}
Completely similar to Theorem \ref{thm_bound_murnaghan_aguiar}.
\end{proof}

\vspace{5mm}

\noindent \textbf{For $\boldsymbol{(\alpha,\beta;\gamma)=\big((2),(1);(3)\big)}$:} Then $\overline{\Lb}=\oo(2)\otimes\oo(1)\otimes\oo(3)$ and a non-zero $G$-invariant section of $\overline{\Lb}$ on $\overline{X}$ is given by $\C(\varphi_1+\varphi_2+\varphi)\big)\in\overline{X}\mapsto\varphi_1(v_1)\varphi_1(v_1)\varphi_2(v_2)$. As a consequence,
\[ \overline{X}^{us}(\overline{\Lb})=\{\big(\C v_1,\C v_2,\C(\varphi_1+\varphi_2+\varphi)\big)\in\overline{X}\,|\,\varphi_1(v_1)\varphi_2(v_2)=0\}=\overline{\oo_2}\cup\overline{\oo_3}. \]
Then we take the same $\tau_2$ as before to destabilise $\overline{x}_2$ (still $\mu^{\overline{\Lb}}(\overline{x}_2,\tau_2)=1$), and $\tau_3=\left.\big(-3,-2,\dots,-2\text{ }\right|1,0,\dots,0,-2\big)$ which destabilises $\overline{x}_3$: $\mu^{\overline{\Lb}}(\overline{x}_3,\tau_3)=1$. Finally,
\vspace{-5mm}
\begin{changemargin}{-1mm}{-1mm}
\[ \begin{array}{rcl}
\max_{x\in\pi^{-1}(\overline{x}_3)}\left(-\mu^\Mb(x,\tau_3)\right) & = & 3\lambda_1+2\dps\sum_{k=1}^{n_1-1}\lambda_{k+1}-\mu_1+2\mu_2-2\nu_1+\nu_2-\sum_{k=1}^{n_1-1}\nu_{n_2+k}\\
 & & -2 \dps\sum_{k=1}^{n_1n_2-n_1-n_2+2}\nu_{n_1+n_2-1+k}-3\sum_{k=1}^{n_2-1}\nu_{n_1n_2+1+k}\\
 & & -4\dps\sum_{k=1}^{n_1-1}\nu_{n_1n_2+n_2+k}-5\nu_{n_1n_2}.
\end{array} \]
\end{changemargin}

\begin{theo}
The sequence $\left(a_{\lambda+(2d),\mu+(d)}^{\nu+(3d)}\right)_{d\in\Z_{\geq 0}}$ is constant when
\[ d\geq\max\left(-\lambda_1+\mu_1+\sum_{k=1}^{n_2}\nu_{k+1}-\sum_{k=1}^{n_1}\nu_{n_1n_2+n_1+n_2+1-k}, \max_{x\in\pi^{-1}(\overline{x}_3)}\left(-\mu^\Mb(x,\tau_3)\right) \right). \]
\end{theo}

\begin{proof}
Once again similar to Theorem \ref{thm_bound_murnaghan_aguiar}.
\end{proof}

\bibliographystyle{alpha}
\bibliography{Biblio_these}
\end{document}